\documentclass{amsart}
\usepackage{amsmath,dsfont,amsfonts}
\usepackage{amssymb}
\usepackage{amsthm}
\usepackage{verbatim}

\newtheorem{teo}{Theorem}[section]
\newtheorem{prop}[teo]{Proposition}
\newtheorem{lm}[teo]{Lemma}

\newtheorem{rem}[teo]{Remark}
\newtheorem*{nota}{Notation}
\newcommand{\RR}{{\mathbb{R}}}
\newcommand{\CC}{{\mathbb C}}
\newcommand{\MM}{{\mathbb M}}
\newcommand{\ds}{\displaystyle}
\newcommand{\der}{\partial}
\newcommand{\ep}{\epsilon}

\newcommand{\om} {\Omega}
\newcommand{\omep}{{\omega_\ep}}

\newcommand{\omepn}{{\omega_{\ep_n}}}
\newcommand{\omepnj}{{\omega_{\ep_{n_j}}}}
\newcommand{\uep}{{u_\ep}}
\newcommand{\uepn}{{u_{\ep_n}}}
\newcommand{\vep}{{v_\ep}}
\newcommand{\vij}{{v^{ij}}}
\newcommand{\vepij}{{v_\ep^{ij}}}
\newcommand{\vhk}{{v^{hk}}}
\newcommand{\vephk}{{v_\ep^{hk}}}

\newcommand{\uepnj}{{u_{\ep_{n_j}}}}
\newcommand{\nabhat}{{\widehat{\nabla}}}
\newcommand{\omeprime}{{\omega^\prime_\ep}}
\newcommand{\sigz}{{\sigma_0}}

\newcommand{\nabuzero}{{\nabhat \uzero}}
\newcommand{\uepe}{{u^e_\ep}}
\newcommand{\uepi}{{u^i_\ep}}
\newcommand{\uzero}{{u_0}}

\def\qed{\hfill$\square$\vspace{0.5cm}}    

\numberwithin{equation}{section}

\def\ds{\displaystyle}
\def\d{{\delta}}

\newcommand{\dive}{\text{\normalfont div}}

\begin{document}

\title[Anisotropic elastic inclusions]{
An asymptotic formula for the displacement field in the
presence of small anisotropic elastic inclusions}

\author[E.~Beretta et al.]{Elena~Beretta}
\address{ Dipartimento di Matematica ``G. Castelnuovo''
Universit\`a
di Roma ``La Sapienza'', Piazzale Aldo Moro 5, 00185 Roma, Italy}
\email{beretta@mat.uniroma1.it}
\author[]{Eric~Bonnetier}
\address{
Laboratoire Jean Kuntzmann, Universit\'e Joseph Fourier,
BP 53, 38041 Grenoble Cedex 9, France}
\email{Eric.Bonnetier@imag.fr}
\author[]{Elisa~Francini}
\address{Dipartimento di Matematica "U. Dini",
Viale Morgagni 67A, 50134 Firenze, Italy}
\email{francini@math.unifi.it)}
 \author[]{Anna~L~Mazzucato}
 \address{Mathematics Department, Penn State University, University Park, PA, 16802, U.S.A.}
\email{alm24@psu.edu}

\date{\today}


\begin{abstract}
We derive asymptotic expansions for the displacement at the boundary of a smooth, elastic body in the presence of small inhomogeneities. Both the body and the inclusions are allowed to be anisotropic. This work extends prior work of  CapdeBoscq and Vogelius ({\em Math. Modelling Num. Anal.} 37, 2003) for the conductivity case.
In particular, we  obtain  an asymptotic expansion of the difference between the displacements at the boundary with and without inclusions, under Neumann boundary conditions, to first order in  the measure of the inclusions. We impose no geometric conditions on the inclusions, which need only be measurable sets. The first-order correction contains an elastic moment  tensor $\MM$ that encodes the effect of the inclusions.
In the case of thin, strip-like, planar inhomogeneities we obtain a formula for $\MM$ only in terms of the elasticity tensors, which we assume strongly convex, their inverses, and a frame on the curve that supports the inclusion. We prove uniqueness of $\MM$ in this setting and recover the formula previously obtained by Beretta and Francini ({\em SIAM J. Math. Anal.}, 38, 2006).
\end{abstract}

\maketitle

\section{Introduction}
Let $\om\subset \RR^d$, $d\geq 2$ be a smooth bounded domain representing the region occupied by an elastic body.
Let $\CC_0=\CC_0(x)$ be a smooth background  elasticity tensor in $\Omega$. Let $\omep\subset \Omega$ be a set of measurable small inhomogeneities and let $\CC_1=\CC_1(x)$ be the smooth elasticity tensor inside the inhomogeneities.  Let $\psi\in H^{-1/2}(\der\om)$ represents a traction on $\partial\Omega$ and $U$ the corresponding background displacement field which satisfies the system of linearized elasticity:
\[
    \left\{\begin{array}{rcl}
             \mbox{div}(\CC_0  \nabhat U)& = & 0\mbox{ in }\om \\
             (\CC_0\nabhat U)\nu & = & \psi\mbox{ on }\der\om,
           \end{array}\right.
\]
Let
\[
\CC_\ep=\CC_0\chi_{\om\setminus\omep}+\CC_1\chi_{\omep},
\]
and consider the perturbed displacement field  solution to
\[
    \left\{\begin{array}{rcl}
             \mbox{div}(\CC_\ep  \nabhat \uep)& = & 0\mbox{ in }\om \\
             (\CC_{\ep}\nabhat \uep)\nu & = & \psi\mbox{ on }\der\om.
           \end{array}\right.
\]
One goal of this paper is to obtain an asymptotic formula for $u_{\ep}-U$ on the boundary of $\Omega$ as the measure of $\omep$ approaches zero. The formula we derive generalizes those already available in case of homogeneous isotropic bodies with diametrically small  (see \cite{AKNT}, and, also \cite{AmmariKang}, \cite{AmmariKang2}) or for thin (see \cite{BerettaFrancini})  inhomogeneities. To derive the asymptotic expansion  we follow the approach introduced by Capdeboscq and Vogelius in \cite{CapdeboscqVogelius_genrep} (see also \cite{CapdeboscqVogelius_survey}) for the conductivity equation and we establish a formula in the case of arbitrary elastic tensors $\CC_0$ and $\CC_1$. More precisely  we show that for $y\in\partial\Omega$
\begin{equation}\label{general_rep}
    (\uepn-U)(y)=|\omepn|\int_\om \MM(x)\nabhat U(x):\nabhat N(x,y)d\mu_x+o(|\omepn|).
\end{equation}
along a sequence of configurations $\{\omepn\}$ whose measure tends to 0.
Here, $N$ is the Neumann function corresponding to the operator ${div}(\CC_0  \nabhat \cdot)$,  $\mu$ is a Radon measure, the elastic moment tensor $\MM \in L^2(\Omega, d\mu)$  and $\nabhat U$ represents the symmetric deformation tensor. \\
For particular geometries  like diametrically small or thin inhomogeneities the asymptotic expansion holds  for $(u_{\ep}-U)(y)$ as $\ep\rightarrow 0$ and one can characterize the measure $\mu$ and the tensor $\MM$. In particular,  if $\omep=z+\ep B$, where the center $z\in \Omega$ and $B$ is a bounded domain, then $\mu$ is a Dirac function concentrated at $z$.
If further both $\CC_0$ and $\CC_1$ are homogeneous and isotropic, the tensor $\MM$ can be explicitly computed and carries information about the geometry of $B$ and about the elastic parameters of $\CC_0$ and $\CC_1$ (\cite{AKNT}). \
If $\omep = \{ x \in \Omega, \quad \textrm{dist}(x, \sigma_0) < \ep \}$,
where $\sigma_0$ is a simple smooth open curve in the plane,
$\mu$ reduces to a Dirac measure supported on $\sigma_0$.
If again the phases are isotropic,
$\MM$ can be explicitly determined by the transmission conditions for $u_{\ep}$ \cite{BerettaFrancini}.

In the second part of the paper we analyze the case of thin inhomogeneities in a planar domain in the case  of arbitrary elasticity tensors.\\
In the case of isotropic homogeneous tensors, the idea,  used in \cite{BerettaFrancini} to derive the asymptotic expansion for $u_{\ep}-U$,  is  to apply  fine regularity results for solutions of elliptic systems with discontinuous coefficients by Y.Y. Li and L. Nirenberg~\cite{LiNirenberg} and to use the transmission conditions to derive the tensor $\MM$ which satisfies
\begin{equation} \label{jump_iso}
(\CC_1 - \CC_0) \nabhat \uepi(x ) = \MM(x) \nabhat \uepe(x),
\end{equation}
 whereas $\nabhat \uepi$ and $\nabhat \uepe$
denote the values of the deformation tensor inside and outside the inclusion
at a point $x$ on its boundary.
Note that the deformation tensors are related by transmission conditions
across $\der \omep$.

One of the difficulties we encountered in deriving the expansion in the anisotropic case is the the direct derivation of~(\ref{jump_iso}) from the transmission conditions.
To construct $\MM$, we follow the work of Francfort and Murat \cite{FrancfortMurat} on
the calculation of the effective properties of laminated 2-phase elastic composites.
Indeed, one can view moment tensors as limits of effective tensors as the volume fraction
of one of the phases tends to 0.

The paper is organized as follows. In section 2, we state a general representation formula
of the form~(\ref{general_rep}) for anisotropic elastic inhomogeneities embedded in
an anisotropic background medium.
The asymptotic expansion is proved in section 3.
Properties of the elastic moment tensor $\MM$ are established in section 4.
In section 5, $\omep$ is assumed to be a thin strip-like planar inclusion.
Firstly, relying on the uniform H\"older regularity of $\uep$ and on Meyer's theorem,
we give a direct derivation of the asymptotic expansion similar
to that in~\cite{BerettaFrancini}, under the assumption that there exists
a tensor $\MM$, independent of $\ep$, that satisfies~(\ref{jump_iso}).
Secondly, we prove existence of such $M$, invoking the result of Francfort and Murat mentioned above
~\cite{FrancfortMurat} . Thirdly, we show that the asymptotic expansion of theorem 2.1
coincides with that obtained in theorem~\ref{teo2.1}.
Finally, in the appendix, we recall classical regularity results for the system of
elasticity, and we prove how  Caccioppoli
inequality and Meyer's theorem also hold for the system of elasticity.

\noindent{\bf Acknowledgments:}  The authors wish to thank Michael Vogelius for useful discussions. E. Bonnetier, E. Beretta, and A. Mazzucato acknowledge the support and hospitality  of the Mathematical Sciences Research Institute (MSRI) where part of this work was conducted. Research at MSRI is supported in part by the National Science Foundation  (NSF). The work of A. Mazzucato was partially supported  by NSF grant DMS-0708902 and DMS-1009713. The work of E. Beretta and E. Francini was partially supported by Miur by grant
PRIN 20089PWTPS003

\section{Notations, assumptions and main result}\label{sec1}
Let $\om\subset \RR^d$, $d\geq 2$ be a bounded, smooth domain. For $x\in\der\om$, let us denote by $\nu(x)$ the normal direction to $\der\om$ at point $x$. We use the following notation:

\begin{nota}
 Let $\CC$ be a $4$-th order  tensor, let $A$ and $B$ be $d\times d$
 matrices, and let $u$, $v$ denote vectors in $\RR^d$. We set:
\begin{eqnarray*}
&u\cdot v=\sum_{j=1}^du_jv_j\quad Av=\sum_{j=1}^dA_{ij}v_j\quad \CC A=\sum_{k,l=1}^d \CC_{ijkl}A_{kl}&\\
&\left(\CC A\right)v=\sum_{j,k,l=1}^d\CC_{ijkl}A_{kl}v_j\quad \CC A : B=\sum_{i,j,k,l=1}^d \CC_{ijkl}A_{kl}B_{ij},\\
& |A|=(\sum_{ij}A_{ij}^2)^{1/2}.&
\end{eqnarray*}
Moreover, we denote by $\widehat{A} = (A+A^T)/2$ the
symmetrization of the matrix $A$.
In particular, given a vector valued function $u$ defined in $\om$, we
denote by
$\nabhat u$ the strain $\nabhat u=\frac{1}{2}\left(\nabla u+\left(\nabla u\right)^T\right)$.
\end{nota}

Let $\CC_0\in C^{1,\alpha}(\om)$, for some $\alpha\in(0,1)$, be a fourth order elasticity tensor that satisfies the full symmetry properties:
\begin{equation}\label{symm}
(\CC_0(x))_{ijkl}=(\CC_0(x))_{klij}=(\CC_0(x))_{jikl}\quad \forall\, 1\leq i,j,k,l\leq d\mbox{ and }x\in\om,
\end{equation}
and the strong convexity condition, i.e., there exists a constant $\lambda_0>0$ such that
\begin{equation}\label{strconv}
    \CC_0(x) A: A\geq \lambda_0|A|^2\mbox{ for every }d\times d \mbox{ symmetric matrix }A \mbox{ and }x\in\om.
\end{equation}

Let $\psi\in H^{-1/2}(\der\om)$  satisfying the compatibility condition
\begin{equation}\label{2}
    \int_{\der\om} \psi\cdot R=0,
\end{equation}
for every infinitesimal rigid motion $R$, that is $R(x)=Wx+c$ for some skew-symmetric matrix $W$ and $c\in \RR^d$.

The background displacement field $U\in \tilde{H}(\om)$ is defined as the solution to
\begin{equation}\label{1}
    \left\{\begin{array}{rcl}
             \mbox{div}(\CC_0  \nabhat U)& = & 0\mbox{ in }\om \\
             (\CC_0\nabhat U)\nu & = & \psi\mbox{ on }\der\om,
           \end{array}\right.
\end{equation}
where $\tilde{H}(\om)$ is the space of vector valued functions given by
\[
\tilde{H}(\om)=\left\{u\in H^1(\om;\RR^d)\mbox{ such that}\int_{\der \om}u\ d\sigma=0,\quad \int_{\om}\left(\nabla u-(\nabla u)^T \right)dx=0\right\}
\]

Let $\omep$ denote the a subset of $\Omega$, that contains one or several inhomogeneities.
We assume that $\omep$ is measurable and separated from the boundary, that is
$d(\omep,\der\om)\geq d_0>0$. We also assume that the measure $|\omep| > 0$ tends to 0 as $\ep \to 0$.

Let $\CC_1$ denote the elasticity tensor inside $\omep$. We assume that $\CC_1\in C^\alpha(\om)$ is fully symmetric  and strongly convex, i.e.
\begin{equation}\label{4}
    \CC_1(x) A: A\geq \lambda_0|A|^2, \;\mbox{ for every }d\times d \mbox{ symmetric matrix }A\mbox{ and }x\in\om.
\end{equation}
Let $\CC_\ep$ be the elasticity tensor in the presence of the inhomogeneity
\begin{equation} \label{CepsDef}
\CC_\ep=\CC_0\chi_{\om\setminus\omep}+\CC_1\chi_{\omep},
\end{equation}
and consider the corresponding displacement field $\uep\in\tilde{H}(\om)$ solution to
\begin{equation}\label{5}
    \left\{\begin{array}{rcl}
             \mbox{div}(\CC_\ep  \nabhat \uep)& = & 0\mbox{ in }\om \\
             (\CC_{\ep}\nabhat \uep)\nu & = & \psi\mbox{ on }\der\om.
           \end{array}\right.
\end{equation}

For existence and uniqueness of solutions to (\ref{1}) and (\ref{5}) in $\tilde{H}(\om)$ we refer to \cite{Oleinik}, for example.

\vspace{0.2cm}

Since $d(\omep,\der\om)\geq d_0>0$, there exists a  compact set $K_0$, independent of $\ep$,
such that
\begin{equation}\label{7}
    \omep\subset K_0\subset\om\mbox{ and dist}(\omep,\om\setminus K_0)>d_0/2>0.
\end{equation}

We also introduce the Neumann matrix for the operator $\mbox{div}(\CC_0\nabhat \cdot)$,
i.e. the weak solution to
\begin{equation}\label{8}
    \left\{\begin{array}{rcl}
             \mbox{div}(\CC_0  \nabhat N(\cdot,y))& = & -\delta_y\mathbf{I}_d\mbox{ in }\om \\
             (\CC_0\nabhat N(\cdot,y)\nu & = & -\frac{1}{|\der\om|}\mathbf{I}_d\mbox{ on }\der\om,
           \end{array}\right.
\end{equation}
that satisfies the normalization conditions
\begin{equation}\label{8.5}
    \int_{\der\om}N(x,y) \,d\sigma_x=0,\quad\int_{\om}(\nabla_xN(x,y) -\nabla_xN(x,y)^T)\,dx =0,
\end{equation}
where $\mathbf{I}_d$ is the $d$-dimensional identity matrix.

For the existence of such Neumann matrix and its behavior for $x$ close to $y$ we refer to
\cite{Fuchs} where existence and regularity of the Green's matrix for weakly elliptic systems is considered.

\vspace{0.5cm}

The following result generalizes the compactness result of~\cite{CapdeboscqVogelius_genrep}
to the case of elastic inclusions:
\begin{teo}\label{mainteo}
Let $\omepn$ be a sequence of measurable subsets satisfying (\ref{7}) such that, as $n \to \infty$,
$|\omepn|\to 0$ and
\begin{equation} \label{def_mu}
    |\omepn|^{-1}\chi_{\omepn}dx\to d\mu \mbox{ in the weak}^*\mbox{ topology of } (C(\overline{\om}))^\prime,
\end{equation}
for some regular positive Borel measure $\mu$, such that $\int_\om d\mu=1$.

Given $\psi\in H^{-1/2}(\der\om)$ satisfying (\ref{2}), let $U$ and $\uepn$ denote the
solutions to (\ref{1}) and (\ref{5}) respectively.
There exists a subsequence, not relabeled, and a fourth order tensor $\MM\in L^2(\om,d\mu)$  such that, for $y\in\der\om$,
\begin{equation}\label{9}
    (\uepn-U)(y)=|\omepn|\int_\om \MM(x)\nabhat U(x):\nabhat N(x,y)d\mu_x+o(|\omepn|).
\end{equation}
\end{teo}
We will prove this result in the next section.

\section{Proof of theorem~\ref{mainteo}}\label{sec2}

\subsection{Preliminary estimates}

Let $F\in H^{-1}(\om)$ and $f\in H^{-1/2}(\om)$ satisfying the compatibility conditions
\begin{equation*}
    \int_\om F \,dx=\int_{\der\om}f\,d\sigma_x\mbox{ and }\int_\om F\cdot R \,dx=\int_{\der\om} f\cdot R \,d\sigma_x,
\end{equation*}
 for every infinitesimal rigid motion $R$.

 Let $V$ and $\vep$ in $\tilde{H}(\om)$ solve
 \begin{equation}\label{10}
    \left\{\begin{array}{rcl}
             \mbox{div}(\CC_0  \nabhat V)& = & F\mbox{ in }\om \\
             (\CC_0\nabhat V)\nu & = & f\mbox{ on }\der\om,
           \end{array}\right.
\end{equation}
and
\begin{equation}\label{11}
    \left\{\begin{array}{rcl}
             \mbox{div}(\CC_\ep  \nabhat \vep)& = & F\mbox{ in }\om \\
             (\CC_0\nabhat \vep)\nu & = & f\mbox{ on }\der\om,
           \end{array}\right.
\end{equation}
respectively.

\begin{lm}\label{lemma1}
Let $F \in C^{\alpha}(\overline\Omega)$, with $0<\alpha<1$ and let $0<\eta<1/d$.
There exists a constant $C > 0$, such that
\begin{equation}\label{energygrad}
\|\vep-V\|_{H^1(\om)}\leq C|\omep|^{1/2}\left(\|F\|_{C^\alpha(\om)}
+ \|F\|_{H^{-1}(\om)}+\|f\|_{H^{-1/2}(\der\om)}\right),
\end{equation}
and
\begin{equation}\label{energyl2}
\|\vep-V\|_{L^2(\om)}\leq C|\omep|^{\frac{1}{2}+\frac{1}{d}-\eta}\left( \|F\|_{C^\alpha(\om)} + \|F\|_{H^{-1}(\om)}+\|f\|_{H^{-1/2}(\der\om)}\right).
\end{equation}
\end{lm}

\begin{proof}
We adapt the arguments of~(\cite{CapdeboscqVogelius_genrep}) to the system of elasticity.
Since $V$ and $\vep$ solve (\ref{10}) and (\ref{11}) respectively, for every $w\in H^1(\om)$, we have\\\[
\int_\om\CC_\ep \nabhat(\vep-V):\nabhat w\,dx=\int_\omep(\CC_0-\CC_1)\nabhat V:\nabhat w\,dx.
\]
By choosing $w=\vep-V$ and applying Korn's inequality, we show that
\[
\int_{\om}|\nabla(\vep-V)|^2dx\leq C  \int_\om|\nabhat(\vep-V)|^2dx\leq C|\omep|^{1/2}
\|\nabla V\|_{L^\infty(\omep)} \|\nabla(\vep-V)\|_{L^2(\om)}.
\]
It follows from interior regularity results for the elasticity system
with regular coefficients (see, for example, Theorem 6.III, chapter 2 in \cite{Campanato})
that
  \begin{align}
  \|\nabla V\|_{L^{\infty}(\omep)}&\leq C\left(
\| V\|_{H^1(\om)} + \|F\|_{C^\alpha(K_0)} \right) \nonumber \\
  & \leq C\left( \|F\|_{H^{-1}(\om)}+\|f\|_{H^{-1/2}(\der\om)}
 + \|F\|_{C^{\alpha}(K_0)} \right), \label{gradVest}
 \end{align}
and, hence,
\[\|\vep-V\|_{H^1(\om)}\leq C |\omep|^{1/2}\left(\|F\|_{H^{-1}(\om)}+\|f\|_{H^{-1/2}(\der\om)} + \|F\|_{C^\alpha(K_0)} \right).\]

\vspace{0.1cm}

We also have, for any $w\in H^1(\om)$,
\begin{equation}\label{11.5}
\int_\om\CC_0\nabhat(\vep-V):\nabhat w\,dx=\int_\omep (\CC_0-\CC_1)\nabhat\vep:\nabhat w\,dx.
\end{equation}
Let us select $w\in\tilde{H}(\om)$ as the solution to
\begin{equation*}
    \left\{\begin{array}{rcl}
             \mbox{div}(\CC_0  \nabhat w)& = &V-\vep\mbox{ in }\om \\
             (\CC_0\nabhat w)\nu & = &\frac{1}{|\partial\Omega|}\int_\om(V-\vep)dx\mbox{ on }\der\om,
           \end{array}\right.
\end{equation*}

By the smoothness assumption on $\CC_0$ and by interior regularity estimates (see, for example, Theorem 2.I, chapter 2 in \cite{Campanato}) we have that
\[
\|w\|_{H^2(K_0)}\leq C(\|\vep-V\|_{L^2(\om)}+\|w\|_{H^1(\om)}).
\]
By Korn and Poincar\'{e} inequalities
\[\|w\|_{H^1(\om)}\leq C\|\nabla w\|_{L^2(\om)}\leq C\|\nabhat w\|_{L^2(\om)}\leq \|V-\vep\|_{L^2(\om)}\]
and, hence,
\begin{equation}\label{11.65}
\|w\|_{H^2(K_0)}\leq C\|\vep-V\|_{L^2(\om)}
\end{equation}
By Sobolev Embedding Theorem, we have that $\nabla w\in L^p(K_0)$ for every  $1<p<d^*$ where $d^*=\frac{2d}{d-2}$ for $d>2$ and
$d^*=+\infty$ for $d=2$, and
\begin{equation}\label{11.80}
\left(\int_{K_0}|\nabla w|^pdx\right)^{\frac{1}{p}}\leq C_p \|w\|_{H^2(K_0)}\leq C\|\vep-V\|_{L^2(\om)}.
\end{equation}

Let us choose $q\in(\frac{2d}{d+2},2)$ and $p$ such that $\frac{1}{p}+\frac{1}{q}=1$. Notice that, $p\in(1,d^*)$.
By inserting $w$ into (\ref{11.5}) we obtain
\begin{eqnarray}\label{11.75}
  \int_\om (\vep-V)^2 dx &=& \int_\om \CC_0\nabhat(\vep-V):\nabhat w\,dx= \int_\omep(\CC_0-\CC_1)\nabhat\vep:\nabhat w\, dx \nonumber \\
   &\leq& C\left(\int_\omep |\nabla\vep|^q\right)^{\frac{1}{q}}\left(\int_\omep|\nabla w|^p\right)^{\frac{1}{p}}
   \\
   &\leq& C\left(\int_\omep |\nabla\vep|^q\right)^{\frac{1}{q}}\|\vep-V\|_{L^2(\om)}.
\end{eqnarray}
Now, by H\"{o}lder inequality and (\ref{gradVest}) we get
\begin{eqnarray}\label{11.85}
  \|\nabla\vep\|_{L^q(\omep)} &\leq & \|\nabla(\vep-V)\|_{L^q(\omep)}+  \|\nabla V\|_{L^q(\omep)} \nonumber\\
   &\leq& |\omep|^{\frac{1}{q}-\frac{1}{2}} \|\nabla(\vep-V)\|_{L^2(\omep)}+|\omep|^{\frac{1}{q}} \|\nabla V\|_{L^\infty(\omep)}\\
  \nonumber &\leq & C |\omep|^{1/q}\left(\|F\|_{H^{-1}(\om)}+\|f\|_{H^{-1/2}(\der\om)} +\| F\|_{L^{\infty}(\Omega)}\right)
\end{eqnarray}
A combination of (\ref{11.75}), (\ref{11.80}) and (\ref{11.85}) yields
\[
\|\vep-V\|_{L^2(\om)}\leq C_q|\omep|^{1/q}
\left(\|F\|_{H^{-1}(\om)}+\|f\|_{H^{-1/2}(\der\om)} +\| F\|_{C^{\alpha}(\Omega)}\right)
\]
 Note that since for $q\searrow \frac{2d}{d+2}$,
we have $\frac{1}{q}\nearrow \frac{1}{d} + \frac{1}{2}$.
It follows that given  any $0<\eta<1/d$ there exists a constant $C$ such that
\[
\|\vep-V\|_{L^2(\om)}\leq C|\omep|^{\frac{1}{d}
+\frac{1}{2}-\eta}
\left(\|F\|_{H^{-1}(\om)}+\|f\|_{H^{-1/2}(\der\om)} +\| F\|_{C^{\alpha}(\Omega)}\right)
\]
\end{proof}


\subsection{Definition of the elastic moment tensor}\label{sec3}

Let
\begin{equation}\label{vij}
\vij=\frac{1}{2}\left(\mathbf{e}_i x_j+\mathbf{e}_j x_i\right)- c_{ij}
\end{equation}
where $\mathbf{e}_i$ is the $i$-th coordinate direction and
\[c_{ij}=\frac{1}{2|\der \om|}\int_{\der\om}\left(\mathbf{e}_i x_j+\mathbf{e}_j x_i\right)d\sigma\]
and consider $\vepij\in\tilde{H}(\om)$ solution to the problem
\begin{equation}\label{12}
    \left\{\begin{array}{rcl}
             \mbox{div}(\CC_\ep  \nabhat \vepij)& = & \mbox{div}(\CC_0  \nabhat \vij)\mbox{ in }\om \\
             (\CC_\ep\nabhat \vepij)\nu & = & (\CC_0\nabhat \vij)\nu\mbox{ on }\der\om,
           \end{array}\right.
\end{equation}
Observe now that
\begin{eqnarray*}
\left\|\frac{1}{|\omep|}\chi_{\omep}(\CC_1-\CC_0)\nabhat \vepij\right\|_{L^1(\om)}&\leq&
\frac{1}{|\omep|}\int_{\omep}\left|(\CC_1-\CC_0)\nabhat(\vepij-\vij)\right|dx\\
&+&\frac{1}{|\omep|}\int_{\omep}
\left|(\CC_1-\CC_0)\nabhat\vij\right|dx.
\end{eqnarray*}
By Lemma \ref{lemma1} and recalling that $\CC_0$ and $\CC_1$ are bounded and that
\begin{equation}\label{gradvij}
\nabla\vij=\nabhat \vij=\frac{1}{2}(\mathbf{e}_i\otimes\mathbf{e}_j+\mathbf{e}_j\otimes\mathbf{e}_i),
\end{equation}
we have
\[
\left\|\frac{1}{|\omep|}\chi_{\omep}(\CC_1-\CC_0)\nabla \vepij\right\|_{L^1(\om)}\leq
C.
\]
Hence, possibly extracting a subsequence, we may assume that
\begin{equation}\label{conv2}
|\omepn|\chi_{\omepn}(\CC_1-\CC_0)\nabla v_{\ep_n}^{ij}\rightarrow d\MM_{ijlm},
\end{equation}
in the weak$^*$ topology of $C^0(\overline{\om})$,
where $d\MM_{ijlm}$ is a regular Borel Measures with support in $K_0$.
Let $\Phi\in C^0(\overline{\om})$.
By definition of $d\MM_{ijlm}$, we see that
\begin{eqnarray*}
  &&\left|\int_\om\Phi d\MM_{ijlm}\right| = \left|\lim_{n\to\infty}\frac{1}{|\omepn|}\int_\om\chi_{\omepn}(\CC_1-\CC_0)\nabhat v_{\ep_n}^{ij}\Phi \,dx \right|\\
   &&\leq \underline{\lim}_{n\to\infty}\frac{1}{|\omepn|}\int_\om\chi_{\omepn}\left|(\CC_1-\CC_0)\nabhat (v_{\ep_n}^{ij}-\vij)\right|\left|\Phi\right|\,dx \\
   &&+  \lim_{n\to\infty} \frac{1}{|\omepn|}\int_\om\chi_{\omepn}\left|(\CC_1-\CC_0)\nabhat\vij\right|\left|\Phi\right|\,dx\\
   &&\leq\underline{\lim}_{n\to\infty}\frac{C}{|\omepn|^{1/2}}\left(\int_\om|\nabhat( v_{\ep_n}^{ij}-\vij)|^2\,dx\right)^{1/2}
   \left(\int_\om\frac{1}{|\omepn|}\chi_{\omepn}|\Phi|^2dx\right)^{1/2}\\
   &&+C\left(\int_\om|\Phi|^2d\mu\right)^{1/2}\leq C\left(\int_\om|\Phi|^2d\mu\right)^{1/2}.
\end{eqnarray*}
Hence
\[\Phi\rightarrow \int_\om\Phi dM_{ijlm}\]
is a bounded functional on $L^2(\om,d\mu)$, and
\[\int_\om\Phi dM_{ijlm}=\int_\om\Phi \MM_{ijlm} d\mu\]
for some function $\MM_{ijlm}\in L^2(\om,d\mu)$.
The tensor $\MM$ actually relates to the weak limit of $\uep$, as the next lemma
expresses:
\begin{lm}\label{lemma2}
Let $U$ and $\uep$ denote the solutions to (\ref{1}) and (\ref{5}) for $\psi\in H^{-1/2}(\der\om)$ satisfying the compatibility conditions (\ref{2}).
Let $\omepn$ such that $|\omepn|\rightarrow 0$ be a sequence for which (\ref{7}), (\ref{def_mu}) and (\ref{conv2}) hold.

Then, $\frac{1}{|\omepn|}\chi_{\omepn}(\CC_1-\CC_0)\nabhat u_{\ep_n} dx $ is convergent in the weak$^*$ topology of $(C^0(\overline{\om}))^\prime$ with
\[
\lim_{n\to\infty}\frac{1}{|\omepn|}\chi_{\omepn}(\CC_1-\CC_0)\nabhat u_{\ep_n} dx=
\MM\nabhat U d\mu
\]
\end{lm}

\begin{proof} It suffices to prove that we may extract a subsequence of $\{\omepn\}$ such that
$\frac{1}{|\omepnj|}\chi_{\omepnj}(\CC_1-\CC_0)\nabhat \uepnj dx$ converges to
$\MM\nabhat Ud\mu$. The fact that the limit is independent of the particular subsequence
guarantees that the entire sequence is convergent.

Proceeding as for $\vepij$, we see that
\begin{eqnarray*}
\left\|\frac{1}{|\omepn|}\chi_{\omepn}(\CC_1-\CC_0)\nabhat u_{\ep_n}\right\|_{L^1(\om)}
&\leq&
  \frac{1}{|\omepn|}\int_{\omepn}\left|(\CC_1-\CC_0)\nabhat( u_{\ep_n}-U)\right|dx
\\
&+&
\frac{1}{|\omepn|}\int_{\omepn}\left|(\CC_1-\CC_0)\nabhat U\right|dx\\
&\leq& C\|\psi\|_{H^{-1/2}(\der\om)},
\end{eqnarray*}
hence, possibly extracting a subsequence, that we do not relabel, we may assume that,
for some matrix-valued measure $\eta$,
\[
\frac{1}{|\omepn|}\chi_{\omepn}(\CC_1-\CC_0)\nabhat u_{\ep_n} dx\rightarrow d\eta
\]
in the weak$^*$ topology of $(C^0(\overline{\om}))^\prime$.

We must now show that, for any scalar function $\Phi$,
\begin{equation}\label{dens}
\int_\om\Phi d\eta=\int_\om\Phi \MM \nabhat Ud\mu.
\end{equation}

In order to do this, it is enough to prove that
\begin{equation}\label{contofinale}
  \int_\omep(\CC_0-\CC_1)\nabhat U:\nabhat \vepij\Phi\,dx =
  \int_\omep(\CC_0-\CC_1)\nabhat\uep:\nabhat \vij\Phi\,dx+o(|\omep|),
\end{equation}
because then, by passing to the limit along subsequences of $\omepn$ in (\ref{contofinale}), we get
(\ref{dens}).

Let us notice that, since
\[\left\{\begin{array}{rcl}\mbox{div}\left(\CC_0\nabhat U\right)&=&\mbox{div}\left(\CC_\ep\nabhat \uep\right)\quad\mbox{in}\quad\om\\
(\CC_0\nabhat U)\nu&=&(\CC_\ep\nabhat \uep)\nu\quad\mbox{on}\quad\der\om,\end{array}\right.\]
for every vector valued test function $\Psi$ we have that
\begin{equation}\label{solint1}
\int_\om\CC_0\nabhat U:\nabhat\Psi\,dx=\int_\om\CC_\ep\nabhat \uep:\nabhat\Psi\,dx.\end{equation}
For the same reason
\begin{equation}\label{solint2}\int_\om\CC_0\nabhat \vij:\nabhat\Psi\,dx=\int_\om\CC_\ep\nabhat \vepij:\nabhat\Psi\,dx.\end{equation}
We can calculate
\begin{eqnarray*}
&&\int_\om\!\!(\CC_0-\CC_\ep)\nabhat U:\nabhat \vepij\Phi\,dx-
  \int_\om\!\!(\CC_0-\CC_\ep)\nabhat\uep:\nabhat \vij\Phi\,dx\\
  &&=\int_\om\!\!\left(\CC_0\nabhat U:\nabhat(\vepij\Phi)-\CC_\ep\nabhat \vepij:\nabhat (U\Phi)\right)dx
  \\&&
  -\int_\om\!\!\left(\CC_0\nabhat\vij:\nabhat(\uep\Phi)-\CC_\ep\nabhat\uep:\nabhat (\vij\Phi)\right)dx\\
  &&-\int_\om\!\!\left(\CC_0\nabhat U:(\vepij\otimes\nabla\Phi)-\CC_\ep\nabhat \vepij:(U\otimes\nabla \Phi)\right)dx\\
  &&+\int_\om\!\!\left(\CC_0\nabhat\vij:(\uep\otimes\nabla\Phi)-\CC_\ep\nabhat\uep:(\vij\otimes\nabla\Phi)\right)dx
  \end{eqnarray*}
By (\ref{solint1}) and (\ref{solint2}) and recalling that $\CC_\ep=\CC_0$ in $\om\setminus\omep$, we can write
\begin{eqnarray*}
&&\int_\omep\!\!(\CC_0-\CC_1)\nabhat U:\nabhat \vepij\Phi\,dx- \int_\omep\!\!(\CC_0-\CC_1)\nabhat\uep:\nabhat \vij\Phi\,dx\\
&&\hphantom{mm}=
\int_\om\!\!\left(\CC_\ep\nabhat \uep:\nabhat(\vepij\Phi)-\CC_0\nabhat \vij:\nabhat (U\Phi)\right)dx
\\&&\hphantom{mmm}-\int_\om\!\!\left(
  \CC_\ep\nabhat\vepij:\nabhat(\uep\Phi)-\CC_0\nabhat U:\nabhat (\vij\Phi)\right)dx\\
  &&\hphantom{mmm}-\int_\om\!\!\left(\CC_0\nabhat U:(\vepij\otimes\nabla\Phi)-\CC_\ep\nabhat \vepij:(U\otimes\nabla \Phi)\right)dx\\
  &&\hphantom{mmm} -\int_\om\!\!\left(
  \CC_0\nabhat\vij:(\uep\otimes\nabla\Phi)-\CC_\ep\nabhat\uep:(\vij\otimes\nabla\Phi)\right)dx\\
&&\hphantom{mm}=\int_\om\!\!\left(\CC_\ep\nabhat \uep:\nabhat\vepij \Phi-\CC_0\nabhat \vij:\nabhat U\Phi\right)dx\\
&&\hphantom{mmm}-\int_\om\!\!\left(
  \CC_\ep\nabhat\vepij:\nabhat\uep\Phi-\CC_0\nabhat U:\nabhat \vij\Phi\right)dx\\
&&\hphantom{mmm}-\int_\om\!\!\left(\CC_\ep\nabhat \uep:(\vepij\otimes \nabla\Phi)-\CC_0\nabhat \vij:( U\otimes \nabla\Phi)\right)dx\\
&&\hphantom{mmm}-\int_\om\!\!\left(
  \CC_\ep\nabhat\vepij:(\uep\otimes \nabla\Phi)-\CC_0\nabhat U:(\vij\otimes \nabla\Phi)\right)dx\\  &&\hphantom{mmm}-\int_\om\!\!\left(\CC_0\nabhat U:(\vepij\otimes\nabla\Phi)-\CC_\ep\nabhat \vepij:(U\otimes\nabla \Phi)\right)dx\\
  &&\hphantom{mmm}-\int_\om\!\!\left(
  \CC_0\nabhat\vij:(\uep\otimes\nabla\Phi)-\CC_\ep\nabhat\uep:(\vij\otimes\nabla\Phi)\right)dx
\end{eqnarray*}
By symmetry of the elasticity tensors, the first two lines of last equality  give zero. By rearranging
 the various integral in a suitable way we get,
\begin{eqnarray*}
&&\int_\omep\!\!(\CC_0-\CC_1)\nabhat U:\nabhat \vepij\Phi\,dx-
  \int_\omep\!\!(\CC_0-\CC_1)\nabhat\uep:\nabhat \vij\Phi\,dx\\
&&\hphantom{mm}=\int_\om\!\!\CC_\ep\left(\nabhat \uep -\nabhat U\right):((\vepij-\vij)\otimes \nabla\Phi)dx
\\&&\hphantom{mmm}-\int_\om\!\!\CC_\ep\left(\nabhat \vij-\nabhat\vepij\right):(( U-\uep)\otimes \nabla\Phi)dx\\
&&\hphantom{mmm}+\int_\omep\!\!\left(\CC_1-\CC_0\right)\nabhat U:((\vepij-\vij)\otimes \nabla\Phi)dx
\\&&\hphantom{mmm}-\int_\omep\!\!\left(\CC_1-\CC_0\right)\nabhat \vij:(( U-\uep)\otimes \nabla\Phi)dx
\end{eqnarray*}
By Lemma \ref{lemma1} and by regularity of functions $U$ and $v^{ij}$ in $K_0$, we get (\ref{contofinale}).

\end{proof}

\subsection{End of the proof of Theorem \ref{mainteo}}
Let $\omepn$ as above. By the definition of the Neumann matrix it is easy to see that
\begin{equation}\label{loc}
(\uepn -U)(y)=\int_{\omepn}(\CC_1-\CC_0)\nabhat\uepn :\nabhat N(\cdot,y)dx.
\end{equation}
(See \cite{BerettaFrancini} for details).

Let $K_0\subset\om$ the compact set introduced in (\ref{7}). Given $y\in\der\om$ it is possible to find $\Psi_y\in C^0(\overline{\om})$ a matrix valued function such that $\Psi_y(x)=\nabla_x N(x,y)$ for $x\in K_0$.

Using the previous lemma we get
\begin{eqnarray*}
  (\uepn -U)(y) &=&|\omepn|\int_\om\frac{1}{|\omepn|}\chi_{\omepn}(\CC_1-\CC_0)
  \nabhat\uepn :\Psi_y\,dx\\
   &=&|\omepn|\int_\om \MM\nabhat U:\Psi_yd\mu+o(|\omep|)\\
   &=&|\omepn|\int_\om \MM\nabhat U:\nabhat N(\cdot,y)\,d\mu+o(|\omep|).
\end{eqnarray*}


\section{Properties of the elastic moment tensor} \label{sec.properties}
In this section we prove few basic properties of the elastic moment tensor $\MM$:
\begin{prop}\label{proppol}
The elastic moment tensor $\MM$ has the same symmetry properties of the elasticity tensors $\CC_0$ and $\CC_1$, that is
\[\MM_{ijkl}=\MM_{klij}=\MM_{jikl}, \qquad \mu\mbox{-a.e.}\]
for any choice of indices $i,j,k,l$ between $1$ and $d$.

Moreover, for any symmetric matrix $E$,
\begin{equation}\label{tensorprop}
\CC_0\CC_1^{-1}(\CC_1-\CC_0)E:E\leq \MM E:E\leq (\CC_1-\CC_0)E:E\qquad \mu\mbox{-a.e.}.
\end{equation}
\end{prop}
{\it Proof.}
Firstly, we show that $\MM$ enjoys the same symmetry as the elastic tensors $\CC_1$ and $\CC_0$.
To this end, we recall the following equality,
which was obtained in the proof of lemma~\ref{lemma2}.
\[
\frac{1}{|\omep|}\int_\omep (\CC_1-\CC_0)\nabhat U :\nabhat \vepij\Phi  dx=
\frac{1}{|\omep|}\int_\omep (\CC_1-\CC_0)\nabhat \uep :\nabhat \vij\Phi dx+o(1).
\]
Substituting $U$ and $\uep$ for $\vhk$ and $\vephk$ respectively, we see that
\[
\frac{1}{|\omep|}\int_\omep (\CC_1-\CC_0)\nabhat \vhk :\nabhat \vepij\Phi dx=
\frac{1}{|\omep|}\int_\omep (\CC_1-\CC_0)\nabhat \vephk :\nabhat \vij\Phi dx+o(1).
\]
Recalling (\ref{gradvij}) and by the symmetry of $\CC_0$ and $\CC_1$,
we get on one hand that
\begin{eqnarray*}
  \frac{1}{|\omep|}\int_\omep(\CC_1-\CC_0)\nabhat\vhk\nabhat\vepij \Phi dx&=&
  \frac{1}{|\omep|}\int_\omep\sum_{lmpq}(\CC_1-\CC_0)_{lmpq}\delta_{hp}\delta_{kq}
  \frac{\der(\vepij)_m}{\der x_l} \Phi dx \\
   &=& \frac{1}{|\omep|}\int_\omep\sum_{lm}(\CC_1-\CC_0)_{lmhk}\frac{\der(\vepij)_m}{\der x_l} \Phi dx\\
   \rightarrow&&\int_\om \MM_{ijhk}\Phi d\mu.
\end{eqnarray*}
On the other hand
\begin{eqnarray*}
  \frac{1}{|\omep|}\int_\omep(\CC_1-\CC_0)\nabhat\vephk\nabhat\vij\Phi dx&=&
  \frac{1}{|\omep|}\int_\omep\sum_{lmpq}(\CC_1-\CC_0)_{lmpq}
  \frac{\der(\vephk)_q}{\der x_p}\delta_{il}\delta_{jm}\Phi dx\\
   &=&  \frac{1}{|\omep|}\int_\omep\sum_{pq}(\CC_1-\CC_0)_{ijpq}
  \frac{\der(\vephk)_q}{\der x_p}\Phi dx\\
  \rightarrow&&\int_\om \MM_{hkij}\Phi d\mu.
\end{eqnarray*}
It follows that
\begin{equation} \label{majorsymm}
\MM_{ijhk}=\MM_{hkij}, \qquad  \mu-\text{a.e.}
\end{equation}
To obtain the minor symmetry, we observe that
\[
    \sum_{lm} |\omepn|\chi_{\omepn}(\CC_1-\CC_0)_{lmpq} \frac{\partial  (v_{\ep_n}^{ij})_q}{\partial x_p}
     = \sum_{lm} |\omepn|\chi_{\omepn}(\CC_1-\CC_0)_{mlpq} \frac{\partial  (v_{\ep_n}^{ij})_q}{\partial x_p},
 \]
 since   $\CC_1-\CC_0$ is a totally symmetric $4$th-order tensor.
 Then, from \eqref{conv2}, it follows that
 \[
     \MM_{ijlm} = \MM_{ijml}, \qquad  \mu-\text{a.e.}
 \]
 All the minor symmetries now follow from \eqref{majorsymm}.

For the proof of (\ref{tensorprop}), we follow \cite{CapdeboscqVogelius_genrep}, where the case of the
scalar conductivity equation  was discussed.

We begin by fixing a constant, symmetric matrix $E=[E_{ij}]$. We set
\begin{equation} \label{Vvepid}
       V = \sum_{ij} E_{ij} v^{ij}, \qquad  v_\ep  = \sum_{ij} E_{ij} \vepij,
\end{equation}
where $v^{ij}$ and $\vepij$ are given in (\ref{vij}) and (\ref{12}),
and observe that $V$ solves
\begin{equation} \label{Veq}
  \begin{cases}
   \dive ( \CC_0 \nabhat V) =\dive(\CC_0  E),
   &  \text{ in } \om, \\
   (\CC_0 \nabhat V) \nu  =
   (\CC_0 E)\nu,
   & \text{ on }  \partial \om,
  \end{cases}
\end{equation}
 and, consequently, $v^\ep$ solves
\begin{equation} \label{Vepseq}
  \begin{cases}
   \dive ( \CC_\epsilon \nabhat v_\epsilon) =\dive(\CC_0  E),
   &  \text{ in } \Omega, \\
   (\CC_\epsilon \nabhat v_\epsilon) \nu  =
   (\CC_0 E) \nu,
   & \text{ on }  \partial \om.
  \end{cases}
\end{equation}

We next recall that $\MM$ is obtained as the weak limit
\eqref{conv2}. Therefore, given any function $\Phi \in
C^0(\Bar\Omega)$, we have
\[
    \int_{\om}  \MM_{ijkl} \Phi \, d\mu =
    \frac{1}{|\omega_{\ep}|}  \int_{\omega_{\ep}} \left [
      (\CC_1-\CC_0)  \nabhat \vepij\right]_{kl} \Phi(x) \, dx \ +
    \ o(1),
\]
where $\epsilon$ is an element of the sequence $\{\epsilon_n\}$, and
square brackets indicate components.
We choose $\Phi (x)= E_{ij}\, E_{kl}\phi\chi_{\Bar\Omega}$, where $\phi$ is a positive smooth function, and sum over repeated indices:
\[
  \begin{aligned}
   \int_{\om} \MM E:E \phi d\mu &= \sum_{ijkl} \int_{\om} E_{ij} \,
   \mathbb{M}_{ijkl} \, E_{kl}\phi dx \\
 & = \frac{1}{|\omega_{\ep}|}
 \int_{\omega_{\ep}} \sum_{ijkl} E_{ij} \left [
      (\CC_1-\CC_0) \cdot \nabhat v_{\epsilon}^{ij}\right]_{kl}
    E_{kl}\phi dx  \ + \ o(1).
  \end{aligned}
\]
Recall that $\nabla v^{ij} = \widehat{(e_i\otimes e_j)}$, hence
 from \eqref{Vvepid}, using that $E$ is symmetric, we have
\[
      \nabhat V = \sum_{ij} E_{ij}\, e_i\otimes e_j = E,
\]
so that:
\begin{eqnarray}\label{quattro.cinque}
\int_{\om} \MM E:E \phi\, d\mu  &= &\frac{1}{|\omega_{\ep}|}
 \int_{\omega_{\ep}} \sum_{klpq} (\CC_1-\CC_0)_{klpq} \left( \sum_{ij} E_{ij} \partial_p
        [v_{\epsilon}^{ij}]_q \right):
     [\nabhat V]_{kl} \phi \,dx  \nonumber \\&&+  o(1)  \nonumber \\
&=& \frac{1}{|\omega_{\ep}|} \int_{\omega_\ep}
      (\CC_1-\CC_0)\,
        \nabhat v_{\epsilon} : \nabhat V \phi\, dx   + \ o(1)\nonumber\\
&=&\frac{1}{|\omega_{\ep}|}\int_{\omega_\ep}
      (\CC_1-\CC_0)\,
        \nabhat V : \nabhat V \phi\, dx\nonumber\\
&&+\frac{1}{|\omega_{\ep}|}\int_{\omega_\ep}
      (\CC_1-\CC_0)\,\nabhat( v_{\epsilon}-V) : \nabhat V \phi \,dx + \ o(1)
  \label{Mconvex.1}
\end{eqnarray}
Let us now notice that
\[
  \int_{\omega_\ep}
      (\CC_1-\CC_0)\,\nabhat( v_{\epsilon}-V) : \nabhat V \phi \,dx  = -\int_\om\CC_\ep\nabhat(v_\ep-V):\nabhat(v_\ep-V)\phi\,dx+R
      \]
where, by (\ref{Veq}), (\ref{Vepseq}) and by Lemma \ref{lemma1},
\begin{eqnarray}\label{R}
R&:=&\int_\omep(\CC_1-\CC_0)\,\nabhat( v_{\epsilon}-V) \, :\, \nabhat V \phi \,dx + \int_\om\CC_\ep\nabhat(v_\ep-V):\nabhat(v_\ep-V)\phi\,dx\nonumber\\
     &=&\int_\om(\CC_\ep-\CC_0)\,\nabhat( v_{\epsilon}-V) \, :\, \nabhat V \phi \,dx + \int_\om\CC_\ep\nabhat(v_\ep-V):\nabhat(v_\ep-V)\phi\,dx\nonumber\\
     &=&\int_\om \CC_\ep\nabhat v_\ep:\nabhat(v_\ep-V)\phi\,dx-\int_\om\CC_0\nabhat V:\nabhat(v_\ep-V)\phi\,dx\nonumber\\
     &=&\int_\om \CC_\ep\nabhat v_\ep:\nabhat((v_\ep-V)\phi)\,dx-\int_\om\CC_0\nabhat V:\nabhat((v_\ep-V)\phi)\,dx\nonumber\\
     &&-\int_\om \CC_\ep\nabhat v_\ep:((v_\ep-V)\otimes\nabla\phi\,dx+\int_\om\CC_0\nabhat V:((v_\ep-V)\otimes\nabla\phi\,dx\nonumber\\
     &=&-\int_\om \CC_0\nabhat( v_\ep-V):((v_\ep-V)\otimes\nabla\phi\,dx-\int_\omep\CC_1\nabhat v_\ep:((v_\ep-V)\otimes\nabla\phi\,dx\nonumber\\
     &=&o(|\omep|)
\end{eqnarray}
By inserting the above relation in (\ref{quattro.cinque}) we get
\begin{eqnarray}\label{MEE}
   \int_{\om} \MM E:E \phi \,d\mu &=&\frac{1}{|\omega_{\ep}|}\int_{\omega_\ep}
      (\CC_1-\CC_0)\,
        \nabhat V \, :\, \nabhat V \phi\, dx\nonumber
        \\&-&\int_\om\CC_\ep\nabhat(v_\ep-V):\nabhat(v_\ep-V)\phi\,dx+o(1).
\end{eqnarray}
Since $\CC_\ep$ is strongly convex, and $\phi>0$,
\[
\int_\om\CC_\ep\nabhat(v_\ep-V):\nabhat(v_\ep-V)\phi\,dx\geq 0
\]
and, hence,
\[\int_{\om} \MM E:E \phi \,d\mu \leq\frac{1}{|\omega_{\ep}|}\int_{\omega_\ep}
      (\CC_1-\CC_0)\,
        \nabhat V \, :\, \nabhat V \phi\, dx+o(1)\]
      Since the left-hand side does not depend on $\ep$ we let $\ep\to 0$ and get
 \begin{equation}\label{sopra}
 \int_{\om} \MM E:E \phi\, d\mu \leq\int_{\om}
      (\CC_1-\CC_0)\,
        \nabhat V  : \nabhat V \phi\, d\mu=\int_{\om}
      (\CC_1-\CC_0)\,
        E: E\phi\, d\mu.
 \end{equation}

 Now, by (\ref{R}) and by the fact that $\CC_1$ is strongly convex, we have that
 \begin{eqnarray}\label{dislemma}
&& \int_\om\CC_\ep\nabhat(v_\ep-V):\nabhat(v_\ep-V)\phi\,dx\nonumber\\&&=\int_\omep(\CC_0-\CC_1)\,\nabhat( v_{\epsilon}-V) \, :\, \nabhat V \phi \,dx
 +o(|\omep|)\nonumber\\
 &&=\int_\omep\CC_1\,\nabhat( v_{\epsilon}-V) \, :\CC_1^{-1}(\CC_0-\CC_1)\, \nabhat V \phi \,dx
 +o(|\omep|)\nonumber\\
 &&\leq\left(\int_\omep\CC_1\nabhat( v_{\epsilon}-V):\nabhat( v_{\epsilon}-V)\phi\,dx\right)^{1/2}\cdot\nonumber\\
&& \left(\int_\omep(\CC_0-\CC_1)\nabhat V:\CC_1^{-1}(\CC_0-\CC_1)\nabhat V\phi\,dx\right)^{1/2}+o(|\omep|)
 \end{eqnarray}
 from which it follows that
 \begin{equation}\label{dise}
   \int_\om\CC_\ep\nabhat(v_\ep-V):\nabhat(v_\ep-V)\phi\,dx\leq \int_\omep\CC_1^{-1}(\CC_0-\CC_1)\nabhat V:(\CC_0-\CC_1)\nabhat V\phi\,dx+o(|\omep|).
 \end{equation}
 By inserting (\ref{dise}) into (\ref{MEE}) we get
 \begin{eqnarray*}
 \int_{\om} \MM E:E \phi\, d\mu &\geq& \frac{1}{|\omega_{\ep}|}\int_{\omega_\ep}
      (\CC_1-\CC_0)\,
        \nabhat V  : \nabhat V \phi\, dx \\&&-\frac{1}{|\omega_{\ep}|}\int_\omep\CC_1^{-1}(\CC_0-\CC_1)\nabhat V:(\CC_0-\CC_1)\nabhat V\phi\,dx+o(1)
        \\&&=\frac{1}{|\omega_{\ep}|}\int_\omep \CC_1^{-1}(\CC_1-\CC_0)\nabhat V:\CC_0\nabhat V dx +o(1).
        \end{eqnarray*}
 By letting $\ep\to 0$ we get
 \begin{equation}\label{sotto}\int_{\om} \MM E:E \phi \,d\mu \geq \int_\om \CC_0\CC_1^{-1}(\CC_1-\CC_0)\nabhat V:\nabhat V\phi\, d\mu=
 \int_\om \CC_0\CC_1^{-1}(\CC_1-\CC_0)E:E \phi\,d\mu.
 \end{equation}
 Notice that (\ref{sopra}) and (\ref{sotto}) hold for every positive $\phi$, hence (\ref{tensorprop}) follows.\qed


\section{The case of thin planar inclusions}  \label{sec.thin}

In this section, we specialize to the case of thin inclusions in a
planar domain $\Omega \subset \RR^2$, modeled
as an appropriate neighborhood $\omega_\ep$ of a given simple curve
$\sigz\subset \om$, that is:
\begin{equation} \label{thinomegaepdef}
\omep=\left\{x\in\om\;:\;d(x,\sigz)<\ep\right\}.
\end{equation}

We impose the following conditions on $\sigz$.
We  assume that $\sigz$ is of class $C^3$ and that there exists some $K>0$
such that
\begin{eqnarray}\label{2.1}
\nonumber  d(\sigz,\der\om) &\geq& K^{-1} \\
\|\sigz\|_{C^3} &\leq& K \\
\nonumber K^{-1} \leq \mbox{length}(\sigz)&\leq& K.
\end{eqnarray}
Moreover we assume that for every $x\in\sigz$ there exists two discs $B_1$ and $B_2$ of radius $K^{-1}$, such that
$$
\overline{B}_1\cap\overline{B}_2=\overline{B}_1\cap\sigz=\overline{B}_2\cap\sigz=\{x\}.
$$
The latter assumption guarantees that different parts of $\sigz$ do
not get too close, so that $\omep$ does not self-intersect for small
$\ep$.
We refer to $\sigz$ as the support of $\omega_\ep$.

Let us fix an orthonormal system $(n,\tau)$ on $\sigz$ such that $n$
is a unit normal vector field to the curve and $\tau$ is a unit
tangent vector field. If $\sigz$ is a closed curve, then we
take $n$ to point in the outward direction of the domain it
encloses.

We present a different derivation of  the small volume
asymptotic formula \eqref{9} for the displacement at the boundary in this
case, which makes more explicit the measure and elastic moment tensor $\MM$
that appear in \eqref{9}.

The main result of this section is the following theorem, which is a counterpart to Theorem \ref{mainteo}.

\begin{teo}\label{teo2.1}
Let $\om\subset\RR^2$ be a bounded smooth domain and let
$\sigz\subset\subset\om$ be a simple curve satisfying (\ref{2.1}).
Let  $\uep$ and
$U$ be the solutions to (\ref{1}) and (\ref{5})
respectively.
For every $x\in\sigz$, there exists a fourth order
elastic tensor field $\tilde{\MM}(x)$ such that, for $y\in\der\om$
\begin{equation}\label{2.11}
(\uep-U)(y)=2\ep\int_\sigz
\tilde{\MM}(x)\nabuzero(x):\nabhat N(x,y)\,d\sigz(x)+o(\ep).
\end{equation}
The term $o(\ep)$ is bounded by
$C\ep^{1+\theta}\|\psi\|_{H^{-1/2}(\der\om)}$, for some $0<\theta<1$ and
$C$ depending only on $\theta$, $\om$, $\alpha_0$, $\beta_0$ and
$K$.
\end{teo}

\subsection{Proof of Theorem \ref{teo2.1}}\label{new4}
The proof of the theorem closely follows the proof of the corresponding
result in the isotropic case (see \cite{BerettaFrancini}).
We only detail those steps, where the proof differs from that case.
In the following we set $\uepi={u_{\ep}}_{|_\omep}$ and  $\uepe={u_{\ep}}_{|_{\om\setminus\omep}}$.
We simply use $\uep$ when no confusion can occur.
Firstly, we write $(\uep-U)_{|_{\der\om}}$ in terms of an integral over $\omep$
of the product of $\nabhat u_\ep^i$ and $\nabhat N$.
Secondly, using some regularity estimates for solutions to the elastic system in a laminar domain due to Li and Nirenberg~\cite{LiNirenberg},
we approximate this integral by an integral over a portion of $\der\omep$,
which we rewrite, in a third step, using the transmission conditions and a
tensor $\tilde{\MM}$ satisfying~(\ref{4.9}).
The existence of $\tilde{\MM}$ is proved later in Subsection \ref{ssec.poltensor}.
Finally, taking limits in the resulting expression as $\ep \to 0$ and using
fine regularity estimates for $\uep$ proves the theorem.

{\bf First step.}

We recall  (see (\ref{loc})) that, for $y\in\der\om$
\begin{equation}\label{4.1}
    (\uep-U)(y)=\int_{\omep}\left(\CC_1-\CC_0\right)\nabhat\uepi:\nabhat
    N(\cdot,y)\,dx.
\end{equation}
%
{\bf Second step.}
Let $\beta$ be a constant, $0<\beta<1$, and set
$$
\omeprime=\left\{x+t\, n(x)\,:\, x\in\sigz,\,
d(x,\der\sigz)>\ep^\beta, \,t\in(-\ep,\ep)\right\}.
$$
Notice that if $\sigz$ is a closed simple curve, then $\omeprime=\omep$.


By Theorem 2.1, chapter 2 in \cite{Campanato} combined with Sobolev Embedding Theorem, we have that $\|\nabla U\|_{L^\infty(\omep)}$
and $\|\nabla N(\cdot,y)\|_{L^\infty(\omep)}$ (for  $y\in\der\om$) are bounded uniformly in
$\ep$. Using this fact together with the energy estimate (\ref{energygrad}),
one can easily show as in~\cite{BerettaFrancini} that
%
\begin{equation}\label{omegaprime}
\int_\omep \left(\CC_1-\CC_0\right)\nabhat\uepi:\nabhat
N(\cdot,y)\,dx=\int_\omeprime
\left(\CC_1-\CC_0\right)\nabhat\uepi:\nabhat N(\cdot,y)\,dx+
O(\ep^{1+\beta/2}).
\end{equation}


%
%
Let $\sigma^\prime_\ep$ denote the curve
$$
\sigma^\prime_\ep=\left\{x+\ep\, n(x)\,:\, x\in\sigz,\,\,
d(x,\der\sigz)>\ep^\beta\right\}.
$$
%
%
%
A crucial ingredient, at this point, is a $C^\alpha$ regularity estimates for the gradient of solutions to laminated systems due to Li and Nirenberg (see \cite{LiNirenberg}). Using this estimate and proceeding as in \cite{BerettaFrancini},
%
we can approximate the values of
$\nabla\uep$ in $\omep^\prime$ by its values on
$\sigma_\ep^\prime$, so that
%
\begin{eqnarray}\label{4.8}
&&\!\!\!\!\int_\omeprime\left(\CC_1-\CC_0\right)\nabhat\uepi(x):\nabhat
N(x,y)\,dx\nonumber\\
&&\hphantom{mm}=2\ep\int_{\sigma_\ep^\prime}
\left(\CC_1-\CC_0\right)\nabhat\uepi:\nabhat N(\cdot,y)+
O(\ep^{1+\alpha-\beta(1+\alpha)}),
\end{eqnarray}
for $\beta<\alpha(1+\alpha)^{-1}$.

{\bf Third step.}

From the results of the section \ref{ssec.poltensor}, for every $x\in\sigma_\ep^\prime$, there exists a fourth-order, symmetric tensor $\tilde{\MM}(x)$, independent of $\epsilon$,  such that
\begin{equation}\label{4.9}
\left(\CC_1(x)-\CC_0(x)\right)\nabhat\uepi(x) =\tilde{\MM}(x)\nabhat\uepe(x)
\end{equation}
Inserting (\ref{4.9}) into (\ref{4.8}), we get
\begin{equation}\label{4.11}
\int_{\sigma_\ep^\prime}\left(\CC_1-\CC_0\right)\nabhat\uepi:\nabhat
N(\cdot,y)dx=2\ep\int_{\sigma_\ep^\prime}\tilde{\MM}\nabhat\uepe:\nabhat
N(\cdot,y)dx+O(\ep^{1-\alpha-\beta(1+\alpha)}).
\end{equation}

{\bf Fourth Step.}\newline Now we show that
\begin{equation}\label{4.12}
\|\nabla\uepe-\nabla U\|_{L^\infty(\sigma_\ep^\prime)}\leq C
\ep^\gamma\|\psi\|_{H^{-1/2}(\der\om)}
\end{equation}
for some positive $\gamma$.

Once estimate (\ref{4.12}) is proved, then (\ref{4.11}) holds with
$\nabla U$ instead of $\nabla\uepe$ and (\ref{2.11}) follows
immediately by continuity.

In \cite{BerettaFrancini} the proof of estimate (\ref{4.12}) strongly relies on the special features of a homogeneous and isotropic tensor $\CC_0$. In the present case, we use of a Caccioppoli-type inequality proved in the Appendix.
%
%

Let $2\ep<d<d_{0}/2$ and $\om_{d}^{\ep}=\{x\in \Omega : d(x,\partial (\Omega\backslash \omega_{\epsilon})>d\}$.

Since $\uep-\uzero$ is solution to
\[
\dive\left(\CC_0\nabhat(\uep-\uzero)\right)=0  \quad\mbox{in}\quad\om\setminus\omep,
\]
the regularity assumption on $\CC_0$ implies that $\uep-\uzero\in H^2_{loc}(\om\setminus\omep)$
( see \cite[Theorem 2.I chapter 2]{Campanato}).

Let $k\in\{1,2\}$ and let $\phi^k_{\ep}:=\der_{k} (u_{\ep}- u_0)$. The function $\phi^k$ solves
\[
\dive\left(\CC_0\nabhat\phi_\ep^k\right)=F\quad\mbox{in}\quad  \om\setminus\omep.
\]
with
\[
F=-\dive\left((\der_k \CC_0)\nabhat(\uep-\uzero)\right).
\]
By Caccioppoli inequality (Theorem \ref{cacc} for $\bar u=0$) and by (\ref{energygrad}),
we see that
\begin{eqnarray*}
\|\nabla \phi_\ep^k\|^2_{L^2(\om^\ep_{d/2})}&\leq &\frac{C_1}{d^2}\|\phi_\ep^k\|^2_{L^2(\om^\ep_{d/4})}+\|F\|^2_{H^{-1}(\om^\ep_{d/4})}
\\&=& C\left(\frac{1}{d^2}\|\phi_\ep^k\|^2_{L^2(\om^\ep_{d/4})}+\|\nabla(\uep-\uzero)\|^2_{L^2(\om\setminus\omep)}\right)\\
&\leq& C\|\psi\|^2_{H^{-1/2}}\left(d^{-2}+1\right)\ep\leq C\|\psi\|^2_{H^{-1/2}}d^{-2}\ep.\end{eqnarray*}
and, by  (\ref{energygrad}) again
\[
\|\phi_\ep^k\|_{H^1((\om^\ep_{d/2}))}\leq C\|\psi\|_{H^{-1/2}} d^{-1}\sqrt{\ep}
\]

Theorem \ref{Teo3} applied to $\phi^k$ shows that
\[
\|\nabla \phi^k\|_{L^{2+\eta}}(\om^\ep_d)\leq C\left(\|F\|_{H^{-1,2+\eta}(\om^\ep_{d/2})}+d^{\frac{2}{2+\eta}-1}\|\nabla \phi^k\|_{L^2(\om^\ep_{d/2})}\right).
\]
Notice that, by applying Theorem \ref{Teo3} to $\uep-\uzero$ in $\om\setminus\omep$,
\begin{eqnarray*}
\|F\|_{H^{-1,2+\eta}(\om^\ep_{d/2})}&=&\|(\der_k \CC_0)\nabhat(\uep-\uzero)\|_{L^{2+\eta}(\om^\ep_{d/2})}\\\leq
C\|\nabhat(\uep-\uzero)\|_{L^{2+\eta}(\om^\ep_{d/2})}&\leq &\|\nabla(\uep-\uzero)\|_{L^2(\om^\ep_{d/4})}
\end{eqnarray*}
and, hence,
\begin{eqnarray*}
\|\nabla \phi^k\|_{L^{2+\eta}}(\om^\ep_d)&\leq& C\left(\|\nabla(\uep-\uzero)\|_{L^2(\om^\ep_{d/4})}+d^{\frac{2}{2+\eta}-1}\|\nabla \phi^k\|_{L^2(\om^\ep_{d/2})}\right)\\&\leq& C\left(d^{\frac{2}{2+\eta}-2}\right)\sqrt{\ep}.
\end{eqnarray*}

On the other hand, applying Theorem \ref{Teo3} to $\uep-\uzero$ shows that
\[\|\phi_\ep^k\|_{L^{2+\eta}(\om^\ep_d)}\leq C d^{\frac{2}{2+\eta}-1}\sqrt{\ep}.\]
By Sobolev Embedding Theorem, it follows that, for $k=1,2$,
\begin{equation}\label{esit1}
\|\der_k (\uep-\uzero)\|_{L^\infty(\om^\ep_d)} =\|\phi_\ep^k\|_{L^\infty(\om^\ep_d)}\leq C\left(d^{\frac{2}{2+\eta}-2}\right)\sqrt{\ep}
\end{equation}

Now, let $y\in \partial \sigma_{\epsilon}'$ and let $y_d$ denote
the closest point to $y$ in the set $\om_{d}^{\ep}$. From the
gradient estimates for $u_{\ep}$ and $u_0$ (see \cite{LiNirenberg} and \cite[Prop.3.3]{BerettaFrancini}), we have
 \begin{equation}
\label{3.5} |\nabla u^e_\ep (y) - \nabla u_\ep^e(y_d)|  \leq
Cd^\alpha,
\end{equation}
which yields, by (\ref{esit1}),
\[
 |\nabla (u^e_\ep - u_0) (y)| 
 \leq C(d^\alpha+d^{-2+\frac{2}{2+\eta}}\ep^{1/2}) \,.
 \]
Choosing  $d=\ep^{\frac{1}{2(2+\alpha-\frac{2}{2+\eta})}}$, we get
 \begin{equation}\label{stim}
 |\nabla (u^e_\ep - u_0) (y)| \leq C\ep^{\gamma} ,
 \end{equation}
where $\gamma =
\frac{\alpha}{2\left(2+\alpha-\frac{2}{2+\eta}\right)}$, and hence
$$
\|\nabla (u^e_\ep- u_0 ) \|_{L^\infty(\sigma_{\epsilon}')} \leq C\ep^{\gamma}~~.
$$

We conclude exactly as in \cite{BerettaFrancini} by noticing that the tensor $\MM$
is continuous in $\om$ (see next section).\qed
\begin{rem}
 If the elasticity tensor $\CC_0$ is smoother than $C^{1,\alpha}$ in $\om$, by differentiating again the equation for $\uep-\uzero$ in $\om\setminus\omep$ and using again Caccioppoli inequality we obtain a better exponent $\gamma$ in (\ref{stim}). Moreover, if $\CC_0\in C^\infty$ this result can be extended to any dimension for small neighborhoods of regular hypersurfaces  contained in $\om$.
\end{rem}


\subsection{Construction of the elastic moment tensor}
\label{ssec.poltensor}

To establish the asymptotic formula \eqref{2.11},
one seeks to express the term
$(\CC_0 - \CC_1) \nabhat \uepi(x)$ as a linear function of
$\nabhat\uepe(x)$, for $x \in \sigma_\ep^\prime$ (with the notation introduced in the second step of the previous proof). This linear
dependence defines the elastic moment $\tilde{\MM}$ in this context, which
therefore must satisfy:
\begin{eqnarray} \label{Mqlgeq}
(\CC_0 - \CC_1) \nabhat \uepi(x)
&=& \tilde{\MM}(x) \nabhat \uepe(x)
\end{eqnarray}
Let us also recall  that $n(x)$ and $\tau(x)$ denote the normal and tangential directions to
$\sigma_\ep^\prime$ at point $x$.

From the discussion in the previous subsection, it suffices to obtain $\tilde{\MM}$
along $\sigma_\ep^\prime$, away from the vertices of the curve (more precisely,
outside a disk of radius $\epsilon^\beta$ centered at the vertices).

We introduce a coordinate system in $\omega^\prime_\ep$, adapted to the
geometry of the problem, as follows. Let $\tau$ be a unit tangent vector
field along the curve $\sigma'$ that supports~$\omega_\ep$, for
instance the velocity vector field using arclength, and let $n$ be a
vector field normal to $\sigma_0$ so that $\{\tau,n\}$ forms a frame
field on $\sigma_0$.

We employ again the frame field $(\tau,n)$ along $\sigma^\prime_0$, and extend it to $\omega^\prime_\ep$ by setting $\tau(x')$ and $n(x')$ to be constant along
the segment $\{x=x' + h n(x'), \; 0\leq h \leq \ep \}$.
By construction  a global coordinate system in $\omega'_\ep$
is given by $x = x' + h n(x')$, where $x'$
is a point along the curve  $\sigma^\prime_0$  and  $0\leq h<\ep$.
We can extend this coordinate system smoothly to part of the boundary,
that is, for  $h=\ep$.

Since the part of the boundary $\omega_\ep$ above $\sigma'_0$ is given
by the graphs of two smooth functions above the same curve,
by the regularity  results recalled in the previous section (see \cite{LiNirenberg}),
the interior and exterior strain fields are regular
and satisfy the following transmission conditions {\em pointwise\/}  on
 $\sigma_\ep^\prime$:
\begin{subequations}    \label{transcond}
 \begin{equation}
\nabla u^e_\ep(x) \tau(x) = \nabla u^i_\ep(x) \tau(x)
\label{transcondtang}
\end{equation}
\begin{equation}
\CC_0 \nabhat\uepe(x)  n(x) = \CC_1 \nabhat\uepi(x)
n(x).
\label{transcondnorm}
\end{equation}
\end{subequations}

A direct determination of a 4th-order tensor $\tilde{\MM}$  that satisfies
\eqref{Mqlgeq} directly
from the above relations turns out to be rather complicated.
One is led to invert a linear system, the determinant
of which is not easily seen to be non zero when $\CC_0$ and $\CC_1$
are anisotropic tensors, unless additional assumptions are made.

Instead, we follow a construction due to
G. Francfort and F. Murat~\cite{FrancfortMurat}.
In their work, effective elastic properties of laminate composites
are calculated in terms of the geometric parameters of the phase layers
(see also~\cite{Milton}, P. 168). This construction precisely relies on
the transmission conditions \eqref{transcond}.
We fix a point $x'\in \sigma'_0$, and  denote  the values of
the interior and exterior strains respectively by
\[e^i=\nabhat\uepi(x)\quad\mbox{and}\quad e^e=\nabhat\uepe(x).\]

\begin{prop}
There exists a 4-th order tensor $\tilde{\MM}$ such that
\begin{equation}\label{richiesta}(\CC_0-\CC_1)e^i=\tilde{\MM}e^e,\end{equation}
where $\tilde{\MM}$ depends on $\CC_0$, $\CC_1$, $n$, and $\tau$, but
not on $e^e$ and $e^i$.
\end{prop}
{\it Proof.}
First of all let us notice that, by the first transmission condition
(\ref{transcondtang}) we can write
\begin{equation}
\label{jump_tang}
e^e\;=\; e^i + \delta \otimes n + n \otimes \delta, \qquad
\textrm{for some}\; \d \in \RR^2
\end{equation}
Let $q^{-1}$ denote the (symmetric) linear operator  associated to the
quadratic form on $\RR^2$:
\begin{eqnarray*}
(q^{-1} \zeta)\xi &=& (\CC_1 (\zeta \otimes n)): (\xi \otimes n),
\qquad \zeta, \xi \in \RR^2.
\end{eqnarray*}


Note that, since $\CC_1$ is positive definite, $q^{-1}$ is
invertible since, by (\ref{4}),
\begin{eqnarray*}
\forall\; \zeta\; \in \RR^2, \quad
(q^{-1} \zeta)\zeta &=& (\CC_1(\zeta \otimes n):(\zeta \otimes n)
\;\geq\; \lambda_0 |\zeta|^2,
\end{eqnarray*}
and thus its inverse $q$ is well defined.
Conditions~(\ref{transcondnorm}) and (\ref{jump_tang}) imply that
\begin{eqnarray}\label{asterisco}
\left((\CC_0-\CC_1)e^e)n\right)&=&\CC_1(e^i-e^e)n\\
&=&(\CC_1 (\delta \otimes n + n \otimes \delta))\, n
\end{eqnarray}
On the other hand, for any $\xi \in \RR^2$, we have
\begin{eqnarray*}
2 (q^{-1} \delta)\xi&=&2 (\CC_1 (\delta \otimes n)):(\xi \otimes n)\\
&=&(\CC_1(\delta \otimes n + n \otimes \delta)) :(\xi \otimes n)\\
&=&\CC_1(e^i-e^e):(\xi\otimes n).
\end{eqnarray*}
This can equivalently be written as
\[2 (q^{-1} \delta)\xi=\left(\CC_1(e^i-e^e)n\right)\cdot\xi,\quad\forall \xi \in \RR^2\]
that is, by (\ref{asterisco})
\[
     2 (q^{-1} \delta)=\CC_1(e^i-e^e)n=((\CC_0-\CC_1)\, e^e) n.
\]
From the above equation we deduce that
\[\delta=\frac{1}{2}q\left((\CC_0-\CC_1)e^e n\right)\]
hence, by (\ref{jump_tang})
\[e^i=e^e+\frac{1}{2}q\left((\CC_0-\CC_1)e^e n\right)\otimes n+n\otimes\frac{1}{2}q\left((\CC_0-\CC_1)e^e n\right).\]
We can now conclude that
\[(\CC_0-\CC_1)e^i=(\CC_0-\CC_1)e^e+(\CC_0-\CC_1)\left(q\left((\CC_0-\CC_1)e^e n\right)\otimes n\right),\]
where, in the last expression we used symmetry of tensor $\CC_0-\CC_1$.

Hence, the forth order tensor $\tilde{\MM}$ defined by
\begin{equation}\label{eq:Mdef}
\tilde{\MM}h=(\CC_0-\CC_1)h+(\CC_0-\CC_1)\left(q\left((\CC_0-\CC_1)h n\right)\otimes n\right),
\end{equation}
satisfies~(\ref{richiesta}).
As can be seen by its expression, $\tilde{\MM}$ does not depend on $e^e$ and $e^i$ but only on the elasticity tensors $\CC_0$ and $\CC_1$ and on the directions $n$ and $\tau$.
Moreover, this tensor $\MM$ can be defined for every point $x$ in $\omep^\prime$ and it is continuous with respect to $x$.

\begin{rem}

Assume that $\{\omep\}$ is a sequence of thin strip-like inclusions
as in~(\ref{thinomegaepdef}). Then,
the expansions~(\ref{9}) and (\ref{2.11}) coincide.
In other words, the measure $\MM(x) d\mu_x$ that appears in~(\ref{9})
is precisely $\tilde{\MM}(x) \delta_{\sigz}(x)$, where $\tilde{\MM}(x)$ is defined
in the section~\ref{ssec.poltensor}.

Indeed, given the form of the sets $\omep$, it is immediate to see that
\begin{eqnarray*}
\ds\frac{1}{|\omega_\ep|} \chi_{\omega_\ep}
&\rightharpoonup&
\delta_\sigz.
\end{eqnarray*}
Thus, recalling Lemma~\ref{lemma2}, it suffices to show that
\begin{eqnarray*}
\ds\frac{1}{|\omega_\ep|} \chi_{\omega_\ep}(\CC_1-\CC_0)\nabhat u_{\ep_n}
&\rightharpoonup&
\tilde{\MM}(x) \nabhat U(x) \delta_\sigz,
\end{eqnarray*}
in the weak$^*$ topology of $(C^0(\overline{\om}))^\prime$.
This again is a consequence of the uniform regularity estimates
on $\uep$.
Indeed, let $\phi \in C^0(\overline{\om})$. Using the same notations and the
same analysis as in section~\ref{new4}, we see that
\begin{eqnarray*}
\ds \int_{\omega_\ep} \ds\frac{1}{|\omega_\ep|} (\CC_1-\CC_0)\nabhat u_\ep \phi\,dx
&=&
\ds\int_{\omega_\ep^\prime} \ds\frac{1}{|\omega_\ep|} (\CC_1-\CC_0)\nabhat u_\ep \phi\,dx\\
&&+
\ds\int_{\omega_\ep \setminus \omega_\ep^\prime} \ds\frac{1}{|\omega_\ep|} (\CC_1-\CC_0)\nabhat u_\ep \phi\,dx
\\
&=&
2 \ds\frac{1}{|\sigma_\ep^\prime|}
\ds\int_{\sigma_\ep^\prime} (\CC_1-\CC_0)\nabhat u_\ep^i \phi \,d\sigma_x
+ O(\ep^{\inf(\alpha - \beta(1 + \alpha), \beta)})
\\
&=&
2  \ds\frac{1}{|\sigma_\ep^\prime|}
\ds\int_{\sigma_\ep^\prime} \tilde{\MM}(x)\nabhat u_\ep^e \phi\,d\sigma_x
+ O(\ep^{\inf(\alpha - \beta(1 + \alpha), \beta)})
\\
&\rightarrow&
2 \ds\int_{\sigz} \tilde{\MM}(x)\nabhat U \phi\,d\mu,
\end{eqnarray*}
which proves the claim.
\end{rem}


\begin{rem}
In the case of isotropic materials, an explicit formula for $\mathbb{M}$ can be obtained from \eqref{eq:Mdef}. In fact, in this case $\CC_0$ and $\CC_1$ have the simple form:
\begin{equation} \label{eq:Cisodef}
    \begin{aligned}
         \CC_0 &= \lambda_0 \, I_2\otimes I_2 + 2 \, \mu_0\, I_4, \\
         \CC_1 &= \lambda_1 \, I_2\otimes I_2 + 2 \, \mu_1\, I_4
    \end{aligned}
\end{equation}
where $\lambda_i$, $\mu_i$, $i=0,1$, are the Lam\'e parameters, $I_2$, $I_4$ are the identity elements on $2$ and $4$ symmetric tensors respectively. Hence, as linear maps over symmetric matrices, they are diagonal in any basis.
These tensors are strongly convex if $\lambda_i >0$ and $\lambda_i + \mu_i>0$, $i=0,1$.

Then, it readily follows that the matrix $q$ is given by (see \cite[Equation 4.8]{FrancfortMurat}):
\begin{equation} \label{eq:qdef}
    q = \frac{1}{\mu_1} I_2  - \frac{\lambda_1 + \mu_1}{\mu_1\, (\lambda_1+2 \mu_1)} n\otimes n.
\end{equation}
Combining \eqref{eq:Cisodef} with \eqref{eq:qdef} gives:
\[
   \begin{aligned}
    &\left(q\left((\CC_0-\CC_1)h n\right)\otimes n\right)=
     \frac{\lambda_0-\lambda_1}{\lambda_1 + 2\mu_1} \text{Tr}(h) \, n\otimes n \\
      &\qquad\qquad + \frac{ 2 \, (\mu_0-\mu_1)}{\mu_1} \left[ (h n) \otimes n- \frac{\lambda_1+\mu_1}{\lambda_1 + 2 \mu_1} \left((h n)\cdot n\right) n\otimes n\right]
   \end{aligned}
\]
Lengthy, but straightforward calculations yield:
\begin{equation} \label{eq:Misodef}
    \mathbb{M} h  = a\, \text{Tr}(h) \, I_2 + b\, h + c \, \left((h \tau)\cdot \tau\right) \tau\otimes
     \tau + d\, \left((h n)\cdot n\right) n\otimes n,
\end{equation}
where
\[
 \begin{aligned}
   a &= (\lambda_0-\lambda_1)\, \frac{\lambda_0 + 2\mu_0}{\lambda_1 + 2\mu_1},
   \qquad b =   (\mu_0-\mu_1)\, \frac{\mu_0}{\mu_1},\\
    c &=  (\mu_0-\mu_1)  \frac{2\lambda_1(\mu_1-\mu_0) + \mu_1(\lambda_1-\lambda_0) + 2\mu_1 (\mu_1-\mu_0)}{\mu_1(\lambda_1+2\mu_1)}, \\
    d& = 2\, (\mu_0-\mu_1) \frac{\mu_1\lambda_0 - \lambda_1 \mu_0}{\mu_1(\lambda_1+2\mu_1)}.
 \end{aligned}
\]
As expected, this formula agrees with that obtained in \cite[Theorem 2.1]{BerettaFrancini}.
\end{rem}


\appendix
%
%
%
%
%
%
\section{A Caccioppoli type inequality}
For the sake of completness, we state and prove a Caccioppoli type inequality for solutions of strongly convex systems.
\begin{teo}\label{Teo4}
Let $B_{\rho}$ and $B_{2\rho}$ be two concentric balls contained in  $\om$  and let  
$\bar u$ be any constant vector.

Let  $u\in H^1(\om)$ be   solution to
\[
\nabla \cdot\left(\CC \nabhat u\right)=f\quad\mbox{in}\quad \Omega,
\]
where $\CC\in L^{\infty}(\Omega)$ is a strongly convex tensor and $f\in H^{-1}(\Omega)$. Then
\begin{equation}\label{cacc}
\|\nabla u\|^2_{L^{2}(B_{\rho})}\leq \frac{C_1}{\rho^{2}}\|u-{\bar u}\|^2_{L^2(B_{2\rho})}+C_2\|f\|^2_{H^{-1}(B_{2\rho})}
\end{equation}
\end{teo}
\begin{rem}Note that ${\bar u}= (u)_{B_{2\rho}}:=\frac{1}{|B_{2\rho}|}\int_{B_{2\rho}} u\,dx$ gives the minimum value for $\|u-{\bar u}\|^2_{L^2(B_{2\rho})}$.\end{rem}

{\it Proof of Theorem \ref{Teo4} }

Let $\theta\in C^{\infty}_0(B_{2\rho})$ such that
\[
0\leq\theta\leq 1, \quad \theta=1 \text{ in } B_\rho \text{ and } |\nabla\theta| \leq \frac{C}{\rho}.
\]

Let $v\in H^1(B_{2\rho}, \RR^2)$ and observe that
\[\nabla (\theta^2 v)= \theta \nabla (\theta v)+\theta\nabla\theta\otimes v.
\]

By assumption $ u$ is a weak solution of
\[
\nabla \cdot\left(\CC \nabhat u\right)=f\quad\mbox{in}\quad B_{2\rho}
\]
hence  we have
\[
a_0(u,\phi) =-<f,\phi> \,\,\forall  \phi \in H^1_0(B_{2\rho})
\]
where
\[
a_0(u,\phi)=\int_{B_{2\rho}}\CC\nabhat u :\nabhat \phi.
\]
Let
$\psi=\theta (u-\bar u)$ and let $\phi=\theta \psi$.
We have that
\begin{eqnarray*}
&&a_0(\psi,\psi)=
\int_{B_{2\rho}} \CC\nabhat \psi:\nabhat \psi \\&&=\int_{B_{2\rho}}[ \CC(\theta\nabhat (u-\bar u)+(u-\bar u)\otimes\nabla\theta):(\theta\nabhat (u-\bar u)+(u-\bar u)\otimes\nabla\theta)]
\end{eqnarray*}
After some straightforward calculations one gets
\[
a_0(\psi,\psi)=a_0(u,\phi)+\int_{B_{2\rho}}\CC\left(\nabla\theta\otimes(u-\bar u)\right):\left(\nabla\theta\otimes(u-\bar u)\right) =<f,\phi>+I
\]
We have
$$
|I|\leq C_3\|\nabla\theta\otimes(u-\bar u)\|^2_{L^2(B_{2\rho})}\leq \frac{C_3}{\rho^2}\|u-\bar u\|^2_{L^2(B_{2\rho})}
$$
and by Young inequality
$$
|<f,\phi>|\leq \|f\|_{H^{-1}(B_{2\rho})}\|\nabla\phi\|_{L^2(B_{2\rho})}\leq \delta \|\nabla\phi\|^2_{L^2(B_{2\rho})}+c(\delta)\|f\|^2_{H^{-1}(B_{2\rho})}
$$
So, we get
\[
a_0(\psi,\psi)-\delta\|\nabla\phi\|^2_{L^2(B_{2\rho})}\leq \frac{C_3}{\rho^2}\|u-\bar u\|^2_{L^2(B_{2\rho})}+c(\delta)\|f\|^2_{H^{-1}(B_{2\rho})}
\]
By the strong convexity of $\CC$ and by Korn inequality applied to the test function $\psi$, we have
\[
a_0(\psi,\psi)\geq \gamma \|\nabla\psi\|^2_{L^2(B_{2\rho})}
\]
which gives
\begin{eqnarray*}
 \frac{C_3}{\rho^2}\|u-\bar u\|^2_{L^2(B_{2\rho})}+c(\delta)\|f\|^2_{H^{-1}(B_{2\rho})}&\geq& \gamma \|\nabla\psi\|^2_{L^2(B_{2\rho})}-\delta \|\nabla\phi\|^2_{L^2(B_{2\rho})}\\&\geq& \gamma \|\nabla\psi\|^2_{L^2(B_{2\rho})}-\delta  \|\nabla(\theta\psi)\|^2_{L^2(B_{2\rho})}
 \end{eqnarray*}
Hence, since $\theta\leq 1$ and $|\nabla\theta| \leq \frac{C}{\rho}$,
\begin{eqnarray*}
&& \frac{C_3}{\rho^2}\|u-\bar u\|^2_{L^2(B_{2\rho})}+c(\delta)\|f\|^2_{H^{-1}(B_{2\rho})}\\
&&\geq  \gamma \|\nabla\psi\|^2_{L^2(B_{2\rho})}-2\delta
  \|\nabla\psi\|^2_{L^2(B_{2\rho})}-\frac{2\delta}{\rho^2} \| \psi\|^2_{L^2(B_{2\rho})}\\
  &&\geq\gamma \|\nabla\psi\|^2_{L^2(B_{2\rho})}-2\delta
  \|\nabla\psi\|^2_{L^2(B_{2\rho})}-\frac{2\delta}{\rho^2} \| u-\bar u\|^2_{L^2(B_{2\rho})}
 \end{eqnarray*}
Choosing $\delta$ small enough, using the fact that $\theta\leq 1$, that $|\nabla\theta| \leq \frac{C}{\rho}$ and the definition of $\psi$ we finally get
\[
 \|\nabla u\|^2_{L^2(B_{\rho})}= \|\nabla\psi\|^2_{L^2(B_{\rho})}\leq  \|\nabla\psi\|^2_{L^2(B_{2\rho})}\leq \frac{C_1}{\rho^{2}}\|u-\bar u\|^2_{L^2(B_{2\rho})}+C_2\|f\|^2_{H^{-1}(B_{2\rho})}.
 \]

\qed
\section{A Meyer's type result}
We state a generalization of Meyer's
theorem concerning the regularity of solutions to systems with
bounded coefficients. For $\eta>0$, define $H^{1,2+\eta}(\om)$ by
$$H^{1,2+\eta}(\om):=\bigg\{ u \in L^{2+\eta}(\om), \nabla u \in
L^{2+\eta}(\om) \bigg\}$$ and let $H^{-1,2+\eta}(\om)$ be its
dual. Introduce
$$H_{\mbox{loc}}^{1,2+\eta}(\om):=\bigg\{ u \in H^{1,2+\eta}(K), \forall K \subset \subset \Omega \bigg\}.$$
\begin{teo}\label{Teo3}
There exists $\eta>0$ such that if $u\in H^1(\om) $ is  solution to
$$
\nabla \cdot\left(\CC \nabhat u\right)= f\quad\mbox{in}\quad\om,\\
$$
where $\CC\in L^{\infty}(\om)$ is a strongly convex tensor and
$f\in H^{-1,2+\eta}(\om)$ then $u\in
H_{\mbox{loc}}^{1,2+\eta}(\om)$ and given
$B_{\rho}$ and $B_{2\rho}$ concentric balls contained in $\om$,
\begin{equation}
\|\nabla u\|_{L^{2+\eta}(B_{\rho})}\leq C(\|f\|_{H^{-1,2+\eta}(B_{2\rho})}+\rho^{\frac{2}{2+\eta}-1}\|\nabla u\|_{L^2(B_{2\rho})}).
\end{equation}
\end{teo}

Theorem \ref{Teo3} has been proved by Campanato in \cite{Campanato} in the case
of strongly elliptic systems. From the proof of this result in
\cite[Chapter II, section 10]{Campanato}, it is clear that the result can be extended to more general systems provided the Caccioppoli type inequality (Theorem \ref{Teo4}) holds.

\bibliographystyle{plain}

\end{document}